\documentclass{amsart}

\usepackage{epsfig}
\usepackage{amsmath}
\usepackage{amssymb}
\usepackage{amscd}
\usepackage{graphicx}
\usepackage{color}
\usepackage{verbatim}

\newtheorem{thm}{Theorem}[section]
\newtheorem{lem}[thm]{Lemma}

\newtheorem{conj}[thm]{Conjecture}

\newtheorem{cor}[thm]{Corollary}
\newtheorem{defn}[thm]{Definition}
\newtheorem{exam}[thm]{Example}

\theoremstyle{remark}
\newtheorem{rem}[thm]{Remark}

\def \na {\mathbb N}

\def \Z {\mathbb Z}

\def \sq {sequence}
\def \xt {$(X,T)$}
\def \ys {$(Y,S)$}
\def \gt {$(G,\tau)$}

\def \tl {topological}
\def \im {invariant measure}

\def \ds {dynamical system}

\def \dens {\mathsf{dens}}
\def \Per {\mathsf{Per}}
\def \Aper {\mathsf{Aper}}
\def \mmu {\boldsymbol{\mu}}

\numberwithin{equation}{section}

\begin{document}

\title[Odometers and Toeplitz systems revisited]{Odometers and Toeplitz systems revisited in the context of Sarnak's conjecture}

\author{Tomasz Downarowicz and Stanis\l aw Kasjan}

\address{Tomasz Downarowicz: Institute of Mathematics of the Polish Academy of Science, \'Sniadeckich 8, 00-956 Warszawa, Poland and Institute of Mathematics and Computer Science, Wroclaw University of Technology, Wybrze\.ze Wyspia\'nskiego 27, 50-370 Wroc\l aw, Poland}
\email{downar@pwr.wroc.pl}

\address{Stanis\l aw Kasjan: Faculty of Mathematics and Computer Science, Nicolaus Copernicus University, 12/18 Chopin street, 87-100 Toru\'n, Poland}
\email{skasjan@mat.umk.pl}

\subjclass[2010]{Primary: 37B05; Secondary: 37B10, 37A35, 11Y35.}
\keywords{Odometer, Toeplitz flow, Almost 1-1 extension, M\"obius function, Sarnak's conjecture, Entropy}

\begin{abstract}
Although Sarnak's conjecture holds for compact group rotations (irrational rotations, odometers), it is not even known whether it holds for all Jewett-Krieger models of such rotations. In this paper we show that it does, as long as the model is at the same a topological extension. In particular, we reestablish (after \cite{Lem}) that regular Toeplitz systems satisfy Sarnak's conjecture, and, as another consequence,
so do all generalized Sturmian subshifts (not only the classical Sturmian subshift). We also give an example of an irregular Toeplitz subshift which fits our criterion. We give an example of a model of an odometer which is not even Toeplitz (it is weakly mixing), hence does not fit our criterion. However, for this example, we manage to produce a separate proof of Sarnak's conjecture. Next, we provide a class of Toeplitz sequences which fail Sarnak's conjecture (in a weak sense); all these examples have positive entropy. Finally, we examine the example of a Toeplitz sequence from \cite{Lem} (which fails Sarnak's conjecture in the strong sense) and prove that it has positive entropy, as well (this proof has been announced in \cite{Lem}). 

\noindent This paper can be considered a sequel to \cite{Lem}, it also fills some gaps of \cite{survey}.
\end{abstract}

\maketitle

\section{Introduction}
This note results from the discussions the authors held with Mariusz Lema\'nczyk about
topological (in particular symbolic) models of odometers in the
context of Sarnak's conjecture. We refer the readers to the recent paper of Lema\'nczyk
et al \cite{Lem} for results concerning Sarnak's conjecture for Morse systems,
where some indispensable facts concerning regular Toeplitz subshifts were obtained partly
independently and partly jointly, and inspired this work.

First of all, it has been discovered that although regular Toeplitz
subshifts are the best known symbolic models (the precise meaning of a ``model'' will be given in the next section)
of odometers, there are also other possibilities, the existence of which was not fully realized before. In the preceding work of the first author \cite{survey} there are some erroneous statements about regularity, and this note fixes them; there exist \emph{irregular} Toeplitz models of their underlying odometers.
For completeness, we also give examples of models which are not even
Toeplitz (in fact \tl ly mixing). Sarnak's conjecture can be shortly proved
for both regular and irregular Toeplitz models, as long as they are \tl\ extensions of
the modeled odometers. The method uses the sole property of the M\"obius function,
that it is orthogonal to any periodic \sq, otherwise it relies on an easy spectral
argument. As a digression, we apply a similar spectral method (but a different property
of the M\"obius function) to prove Sarnak's conjecture for isomorphic extensions of other
equicontinuous systems, in particular for generalized Sturmian subshifts. Our method fails
for other (e.g. \tl ly weakly mixing) models of equicontinuous systems and the case remains a
challenge. Nonetheless, we are able to successfully apply it to our \tl ly mixing example.
By this occasion we also give relatively simple examples of Toeplitz subshifts
which massively fail Sarnak's conjecture (at many points including Toeplitz \sq s),
showing that just being a union of periodic \sq s is insufficient. We also copy
from \cite{Lem} an example in which a Toeplitz \sq\ fails the conjecture
so bad that the limit inferior of the absolute values of the averages is positive.
As one might expect, the counterexamples have positive entropy, of which we give
detailed proofs (which are skipped in \cite{Lem}).

\section{Preliminaries on Toeplitz systems}

The notation and terminology of this section is consistent with that of \cite{survey}, where we also refer
for references to earlier papers. To make this paper selfcontained, the most crucial
definitions will be repeated. By a \emph{scale} we mean
an increasing \sq\ of positive integers $H=(p_k)_{k\ge 1}$ such that $p_k|p_{k+1}$ for
every $k$. The \emph{adding machine} with scale $H$ is the \tl\ group $G$ obtained as the
inverse limit of the cyclic groups $\Z_{p_k}=\Z/p_k\Z$:
$$
G=\overset{\longleftarrow}{\lim_k}\,\Z_{p_k}.
$$
By an \emph{odometer} we will mean the \tl\ \ds\ \gt\, where $\tau$ is the homeomorphism
$g\mapsto g+\mathbf 1$ of $G$ into itself, where $\mathbf 1=(1,1,\dots)$ is the \tl\ generator
of $G$. The odometer is minimal, equicontinuous and zero-dimensional, and the conjunction of
these three properties characterizes odometers among \tl\ \ds s. Odometers are uniquely ergodic,
with the Haar measure $\lambda$ being the unique \im. For us, an \emph{odometer} also means
the ergodic system $(G,\lambda,\tau)$. By a \emph{(Jewett--Krieger) model} of an ergodic
system we will understand any strictly ergodic (minimal and uniquely ergodic) \tl\ \ds\
isomorphic (for its unique \im\footnote{Convention: \emph{isomorphism} is measure-theoretic while \emph{conjugacy} is \tl.}) to the given ergodic system. Note that there may
(and usually do) exist mutually not conjugate models of the same ergodic
system. For instance, an odometer is a model for itself, but there are other models as well, for
example some symbolic systems over finite alphabets (which are never equicontinuous, so they cannot
be conjugate to the odometer).

We will be mostly concerned with Toeplitz systems\footnote{In the literature they are often called ``Toeplitz flows''; we find this notation confusing, as they are discrete time systems.},
understood slightly more generally than usually, i.e., we will not require that they are symbolic, nonetheless, all our examples will be symbolic. Toeplitz systems can be defined in a multitude of ways,
which is captured in the theorem below. The relevant definitions and proofs can be found in \cite{survey}
(for the reader's convenience, the definitions are also given in footnotes).

\begin{thm}\label{uno}
The following conditions are equivalent for a \tl\ \ds\ \xt. A system satisfying them is called a
\emph{Toeplitz system}.
\begin{enumerate}
	\item \xt\ is the orbit closure of a regularly recurrent point\footnote{A point $x$ is \emph{regularly
	recurrent} if, for every open $U\ni x$, the set of return times to $U$ contains an arithmetic
	progression $n\Z$.};
	\item \xt\ is a minimal almost 1-1 extension\footnote{The factor map $\pi:X\to G$ is \emph{almost 1-1}
	if $\pi^{-1}(\pi(x))=\{x\}$ holds on a residual subset of $X$.} of an odometer \gt;
	\item \xt\ is a semicocycle extension\footnote{A \emph{semicocycle} is a function
	$f:G\to K$ into a compact space, which is continuous on a residual subset of $G$. Let $F$ denote the
	multifunction defined by the closure of the graph of $f$. Let $X_F=\{x\in K^\Z: (\exists g\in G)(\forall 	
	n\in\Z)\ x(n)\in F(g+n)\}$. This set is closed and shift-invariant, and has a unique minimal subset which
	we denote by $X_f$. $X_f$ with the action of the shift is called the \emph{semicocycle extension} (associated	with $f$).} of an odometer $(G,\tau)$.
\end{enumerate}
\end{thm}
The conditions (1) and (2) are additionally related:

\begin{thm}\label{dos}
If $\pi$ is the almost 1-1 factor map from a Toeplitz system \xt\ to an odometer \gt\ then $x\in X$
is regularly recurrent if and only if $\pi^{-1}(\pi(x))=\{x\}$.
\end{thm}

In Theorem \ref{uno}, the odometer \gt\ appearing in (2) is the \emph{maximal equicontinuous factor}\footnote{Any other equicontinuous factor of \xt\ factors through \gt.} of \xt. The adding machine $G$ appearing in (3) (with the action of $\tau$) is the maximal equicontinuous factor of \xt\ if and only if the semicocycle has the additional property of being \emph{invariant under no rotations}\footnote{A semicocycle $f$ on an adding machine $G$ is \emph{invariant under no rotations} if $F\circ(\cdot+g)=F\implies g=\mathbf 0$.}. Otherwise, the maximal equicontinuous factor of \xt\ acts on a quotient group of $G$\footnote{This quotient group is $G/H$, where $H=\{g: F\circ(\cdot+g)=F\}$.} and \gt\ is not even a factor of \xt\ (let alone maximal equicontinuous). We can strengthen condition (3) as follows:

\begin{thm}\label{tres}
Every Toeplitz system is conjugate to a semicocycle extension of its maximal equicontinuous factor.
\end{thm}

\medskip
An important class of Toeplitz systems, called \emph{regular}\footnote{The coincidence of this term with ``regular recurrence'' is incidental: the latter term was coined independently from those who invented ``regular Toeplitz systems''.}, is described by the conditions given below:

\begin{thm}\label{cinco}
Let $\pi$ be the almost 1-1 factor map from a Toeplitz system \xt\ to an odometer \gt. The following conditions are equivalent:
\begin{enumerate}
	\item The set of points $g\in G$ such that $\#\pi^{-1}(g)=1$ has full measure $\lambda$;
	\item The set of regularly recurrent points in \xt\ has full measure for every \im\ on $X$;
	\item The set of discontinuities $D_{\!f}$ of the relevant semicocycle $f$ on $G$ has $\lambda$
	measure zero.
\end{enumerate}
\end{thm}

The following is obvious by (1):
\begin{thm}\label{siete}
A regular Toeplitz system \xt\ is strictly ergodic and it is isomorphic to its maximal equicontinuous factor \gt, and the isomorphism is provided by the almost 1-1 \tl\ factor map $\pi$.
\end{thm}
In particular, a regular Toeplitz system is a model for the underlying odometer. Unless $\pi$ is 1-1 everywhere, \xt\ is not conjugate to \gt.

In \cite{survey} Theorem 13.1 (5) and (6) it is claimed that a strictly ergodic Toeplitz system
isomorphic to its maximal equicontinuous factor is necessarily regular. Moreover, it is claimed
that strict ergodicity need not be assumed if an isomorphism exists for some \im. Unfortunately, the statement (even with strict ergodicity assumed) is false. Relevant counterexample is provided in the following sections.
\medskip

Let us return to the general case. In symbolic systems regularly recurrent points are called \emph{Toeplitz \sq s}\footnote{Regular recurrence takes on the form $(\forall n\in\Z)(\exists{p\in\na})(\forall m\in\Z)\ x(n)=x(n+mp)$.}. Toeplitz subshifts (and Toeplitz \sq s) were the first examples of almost 1-1 extensions of odometers and they are the most important. In this class we have an additional simplification:

\begin{thm}\label{cuatro}
Every Toeplitz subshift over a finite alphabet $\Lambda$ is conjugate to a semicocycle extension of its maximal equicontinuous factor with the semicocycle taking values in $\Lambda$.
\end{thm}

The advantage of having the finite-valued semicocycle is that its set of discontinuities $D_f\subset G$
is then closed. The elements of Toeplitz subshifts have specific structure, as described below.

\begin{defn} Let \xt\ be a Toeplitz subshift and let \gt\ be the maximal equicontinuous factor
of \xt. For $x\in X$ and $p\in\na$ we denote
$$
\Per_{p}(x)=\{n\in\Z: (\forall m\in\Z)\ x(n)=x(n+mp)\}, \ \ \Aper(x)=\Z\setminus\bigcup_{p\in\na}\Per_{p}(x).
$$
and call these sets the \emph{$p$-periodic part} and \emph{aperiodic part}, respectively.
\end{defn}

The union of periodic parts will not change if we unite over a scale $(p_k)_{k\in\na}$ of the adding machine $G$ (then the union is increasing). It is important to know that $\Per_{p_k}(x)$ and $\Aper(x)$
depend only on $\pi(x)$. Clearly, $x$ is a Toeplitz \sq\ if and only if its aperiodic part is empty.

Recall that for a set $A\subset\Z$ the \emph{forward and backward densities} of $A$ are defined as $\dens^+(A)=\lim_n \frac1n{\#(A\cap[0,n-1])}$, $\dens^-(A)=\lim_n\frac1n{\#(A\cap[-n,-1])}$,
respectively (provided the limits exist) and in case they coincide we call them the \emph{density} of $A$
and denote by $\dens(A)$.

\begin{thm}\label{irreg}
Let \xt\ be a Toeplitz subshift. Then, for every $p\in\na$, $\dens(\Per_p(x))$ exists and is constant throughout $X$. Let $d=1-\sup_k\dens(\Per_{p_k}(x))$. Then
\begin{enumerate}
  \item $d=\lambda(D_f)$;
  \item $\dens(\Aper(x))\le d$ for every $x\in X$;
	\item $\dens(\Aper(x))=d$ for $\mu$-almost every $x\in X$, for every \im\ $\mu$ on $X$.
\end{enumerate}
\end{thm}
In particular, regularity of \xt\ is equivalent to $d=0$, and to $\dens(\Aper(x))=0$ at every point.

\medskip
Let \xt\ be an irregular Toeplitz subshift. Consider the set of such points $x\in X$ that $\dens(\Aper(x))=d$ (by (3) above, this set has full \im).
Since $\Aper(x)$ is constant throughout every fiber of $\pi$, our set equals $\pi^{-1}(E)$ for some $E\subset G$. For $x\in \pi^{-1}(E)$ we enumerate $\Aper(x)=\{n_i\}_{i\in\Z}$ assuming that the sequence $(n_i)$ is increasing and letting $n_0$ be the smallest
nonnegative element of the \sq. We let $y_x = (x(n_i))_{i\in\Z}\in\Lambda^\Z$ and call it the \emph{aperiodic readout} of $x$. For $g\in E$ we let $Y_g = \{y_x:x\in\pi^{-1}(g)\}$.

\begin{defn}
We say that the Toeplitz subshift \xt\ satisfies the condition SAR (\emph{same aperiodic readouts})
if $Y_g$ is the same for every $g\in E$. We then denote the common space $Y_g$ by $Y$.
\end{defn}
It is easy to see that in this case $Y$ is closed and shift invariant. The following theorem
plays the crucial technical role in most of our examples (for proofs see \cite{survey}):

\begin{thm}\label{tw210}
Let \xt\ be an irregular Toeplitz subshift satisfying the condition SAR. Let $T_f$ be the
skew product acting on $G\times Y$ given by
$$
T_f(g,y)=(\tau(g), S^g(y)),
$$
where $S^g$ equals the shift or the identity, depending on whether $g\in D_{\!f}$ or not, respectively. Then
\begin{enumerate}
	\item There is a bijection between \im s of \xt\ and \im s of $(G\times Y, T_f)$;
	\item Every \im\ on \xt\ is isomorphic to its corresponding \im\ on the skew product;
	\item Every \im\ on the skew product has marginals $\lambda$ on $G$ and some shift-\im\ $\nu$ on $Y$;
	\item Every shift-\im\ $\nu$ on $Y$ appears as the marginal for at least one \im\ on the skew product
	(for example for $\lambda\times\nu$);
	\item The entropy of the skew product with respect to an \im\ equals $d$ times the entropy of the
	corresponding marginal on $Y$.
	\item The \tl\ entropy of the skew product (which equals the topological entropy of \xt) equals $d$
	 times the \tl\ entropy of the shift on $Y$.
\end{enumerate}
\end{thm}
We will refer the the above facts several times.

\begin{comment}
\section{A general scheme of constructing Toeplitz \sq s}

Let $(q_k)_{k\ge 1}$ be an arbitrary \sq\ of integers larger than 1. Then $H=(p_k)$, where
$p_k=q_1q_2\dots q_k$ is a scale of an adding machine $G$. Select another \sq, say $(r_k)$ of
positive integers which satisfy:
\begin{itemize}
	\item $r_1<p_1$,
	\item $r_k<p_k-\sum_{i=1}^{k-1}r_i\frac{p_k}{p_i}$ \ \ (for $k>1$).
\end{itemize}
Let $B_k$ be an arbitrary \sq\ of blocks (over some finite alphabet $\Lambda$) of lengths $r_k$. We create
the one-sided Toeplitz \sq\ $x$ inductively, as follows: in step 1 we place the block $B_1$ at positions $[1,r_1]$ and we repeat it periodically with period $p_1$. This leaves $p_2-r_1q_2$ unfilled positions in the interval $[1,p_2]$. We use the block $B_2$ to fill the leftmost $r_2$ unfilled positions, and we repeat this with the period $p_2$. And so on, it should be clear how to proceed. In the end, we obtain a Toeplitz \sq\ ``built from the blocks $B_k$''.
\end{comment}

\section{Preliminaries on the M\"obius function and Sarnak's conjecture}
The \emph{M\"obius function} denoted by $\mmu$ is defined on positive integers as follows

$$
\mmu(n)=\begin{cases}
\phantom{-}1&\text{if $n=1$,}\\
\phantom{-}0& \text{if $n$ has a repeated prime factor,}\\
\phantom{-}(-1)^r&\text{if $n$ is a product of $r$ distinct primes.}
\end{cases}
$$
This  function has been introduced by A. F. M\"obius in \cite{mobius} to obtain inversion formulas for arithmetic functions \cite{CDM}. It plays an important role in number theory. The reader is referred to the rich literature in that area for more information, let us quote  only two fundamental monographs: \cite{P}, \cite{Wa}.

\begin{defn}
Let $\xi(n)$ and $\eta(n)$ be two bounded complex-valued \sq s over $\na$. We say that these
\sq s are \emph{uncorrelated} if
$$
\lim_n\frac 1n\sum_{i=1}^n\xi(i)\overline\eta(i) = 0.
$$
\end{defn}
One of the intriguing properties of the M\"obius function is its apparent randomness in
the distribution of its values. It is well known that the forward density of square-free numbers
(i.e., of the set $\{n:|\mmu(n)|=1\}$) exists and equals $\frac 6{\pi^2}$ (see \cite{N}, Thm 21.8 and the following Corollary). On the other hand, the densities of positive and negative values are equal implying that $\mmu$ is uncorrelated to the constant \sq\ (see \cite{P}, Thm 5.1).

Moreover, it is uncorrelated to any periodic function (an elementary proof can be found in \cite{GL}), a fact which is connected with the laws of the distribution of primes along arithmetic progressions (see \cite{S}).
Let us remark that a more detailed analysis of this phenomenon, more precisely, of the behavior of the partial sums of the M\"obius function, is an important area of study, connected to many fundamental number theoretical problems, see e.g. classical works  \cite{Wa}, \cite{Sch} and more recent  papers \cite{RaRu},  \cite{HaSu}.

Sarnak \cite{S} conjectures that $\mmu$ is uncorrelated to any \sq\ obtained by reading any continuous function along any orbit in any \tl\ \ds\ with \tl\ entropy zero, as follows:

\begin{conj}
Let \xt\ be a \tl\ \ds\ with \tl\ entropy zero. Let $f:X\to\mathbb C$ be a continuous function. Fix
an $x\in X$ and let $\xi(n)=f(T^n x)$ (for $n\ge 1$). Then $\xi$ and $\mmu$ are uncorrelated.
\end{conj}

The conjecture is known to hold for relatively few types of \ds s, in particular for odometers,
irrational rotations, nilsequences \cite{GT}, horocycle flows \cite{BSZ}. See also \cite{B1}, \cite{B}, \cite{ALR}, \cite{G}, \cite{KL}, \cite{LS}, \cite{MR},  for    other results.  We remark, that validity of Sarnak's conjecture for odometers follows directly from the fact that the M\"obius function is uncorrelated to any periodic \sq. The validity for irrational rotations can be proved by a criterion from \cite{BSZ} or by completely different property of the M\"obius function, discovered by Davenport \cite{Da}. 
It is a folklore fact that the conjecture holds for the classical Sturmian subshift; the proof uses heavily the fact that this subshift has complexity\footnote{the number of words of length $n$ in the subshift} $c(n)=n+1$ (\cite{MH}).
Recently, in \cite{Lem} Sarnak's conjecture has been proved for some Morse subshifts. Every continuous
function on such a system decomposes as the sum of a function depending on the Toeplitz factor and an orthogonal one. Thus the method relies on two ingredients: for continuous functions orthogonal to the Toeplitz factor some specific spectral and disjointness methods are used. To handle the other ingredient the authors simply prove the conjecture for regular Toeplitz subshifts. In this note we extend the latter proof (in fact, we only notice that essentially the same proof applies) to a class slightly larger than regular Toeplitz systems, that of isomorphic extensions of compact monothetic group rotations. This includes some not necessarily regular Toeplitz systems and generalized Sturmian subshifts.

\section{Sarnak's conjecture for isomorphic extensions}

The following fact has been observed jointly by the first author and M. Lema\'nczyk.
\begin{thm}\label{4.1}
Let \xt\ and \ys\ be strictly ergodic \tl\ \ds s, with \im s $\mu$ and $\nu$, respectively,
and let $\pi:X\to Y$ be a \tl\ factor map which is, at the same time, an isomorphism.
If Sarnak's conjecture holds for \ys\ then it also holds for \xt.
\end{thm}

\begin{proof}
Let $f:X\to\mathbb C$ be continuous. Then $f\in L^2(\mu)$ and $f'=f\circ\pi^{-1}\in L^2(\nu)$
($f'$ is defined $\nu$-almost everywhere on $Y$). Since $C(Y)$ is dense in $L^2(\nu)$, there
exists a continuous $g':Y\to\mathbb C$ such that $\int |f'-g'|^2\,d\nu<\epsilon^2$, hence
$\int |f'-g'|\,d\nu<\epsilon$. The function $g=g'\circ\pi$ is continuous on $X$ and $\int |f-g|\,d\mu<\epsilon$. Because in strictly ergodic systems every point is generic\footnote{fulfills the ergodic theorem for every continuous function}, we have $\lim_n\frac1n \sum_{i=1}^n|f(T^ix)-g(T^ix)|<\epsilon$, for any $x\in X$. Now, we write
$$
\left|\frac1n \sum_{i=1}^nf(T^ix)\mmu(i)\right| \le \left|\frac1n \sum_{i=1}^ng(T^ix)\mmu(i)\right|+\frac1n \sum_{i=1}^n|f(T^ix)-g(T^ix)||\mmu(i)|.
$$
The first average on the right hand side equals $\left|\frac1n \sum_{i=1}^ng'(S^iy)\mmu(i)\right|$, where $y=\pi(x)$, and is small for large $n$, because
Sarnak's conjecture holds on \ys. The last average does not exceed, for large $n$, the arbitrarily
small $\epsilon$. Thus the left hand side tends zero with growing $n$.
\end{proof}

In \cite{Lem} the reader will find a slightly different statement, in which \ys\ is assumed \emph{coalescent}\footnote{every endomorphism from the system to itself is an isomorphism}
and the assumption that the isomorphism between \xt\ and \ys\ is realized by the same topological
factor map $\pi$ is dropped (it is then fulfilled automatically). Recall that odometers and other
ergodic group rotations are coalescent.

We now draw conclusions concerning particular types of \tl\ \ds s. It seems that items (1b) and (2) below are new.
Notice that for the classical Sturmian subshift we have obtained a new proof not relying on the exact complexity.
\begin{cor}
Sarnak's conjecture holds for:
\begin{enumerate}
	\item regular semicocycle extensions of any minimal equicontinuous systems\footnote{Notice that semicocycle	extensions can be as well defined on any strictly ergodic system, not necessarily on an odometer. Regularity means that the set of discontinuities of the semicocycle has measure zero.}, in particular
	\subitem {\rm (1a)} regular Toeplitz systems (see also \cite{Lem});
	\subitem {\rm (1b)} generalized Sturmian subshifts\footnote{A classical Sturmian subshift is obtained as the semicocycle extension of the irrational rotation by an angle $\alpha$, where the semicocycle is
	precisely the characteristic function of $[0,\alpha]$. In generalized Sturmian subshifts the
	semicocycle is admitted characteristic function of any nondegenerate subinterval or even a finite
	union of intervals.};
	\item some irregular Toeplitz subshifts as in the Example \ref{exe} below.
\end{enumerate}
\end{cor}

\begin{proof}
To be absolutely clear, let us argue why does the conjecture hold for minimal equicontinuous systems.
By the Halmos--von Neumann Theorem, every such system is uniquely ergodic and the space $L^2(\mu)$ is spanned
by (at most countably many) continuous eigenfunctions. Thus every continuous function can be approximated
in $L^2(\mu)$ (hence also in $L^1(\mu)$) by a finite sum of continuous eigenfunctions. Now, by an argument
as in the preceding proof, it suffices to verify the conjecture for continuous eigenfunctions. But every such function arises as a continuous function defined on either an odometer (if the eigenvalue is rational) or an irrational rotation (otherwise).
\end{proof}

\section{Examples of models of odometers}

\begin{exam}\label{exe} There exist \emph{irregular} Toeplitz subshifts which are strictly ergodic and isomorphic (via the same \tl\ factor map) to their maximal equicontinuous factor odometers (for these
Sarnak's conjecture holds).
\end{exam}

\noindent
\emph{Sketch of the construction.} An explicit example of such a system is generated by the (unilateral) Toeplitz \sq\ described below.
\smallskip

Pick a block $B_1=000...01000...0$ of some length $r_1$ consisting of zeros but one symbol 1 (somewhere).
We place this block $p_1$-periodically (for some $=p_1>r_1$, we also let $q_1=p_1$) (see Figure 1).

\begin{figure}[ht]\label{fig1}
$$
0100\!*\!*\!*0100\!*\!*\!*0100\!*\!*\!*0100\!*\!*\!*0100\!*\!*\!*0100\!*\!*\!*0100\!*\!*\!*0100\!*\!*\!*0100\!*\!*\!*0100...
$$
\caption{\small On this figure $B_1=0100$, $r_1=4$ and $p_1=7$.}
\end{figure}
The unfilled places (the stars) come in blocks of length $Q_1=q_1-r_1$. We pick a block $B_2$ of some length $r_2Q_1$, consisting of zeros but one symbol 1. We write this block into $r_2$ consecutive empty blocks and repeat $p_2$-periodically, where $p_2=q_2p_1$, for some $q_2>r_2$ (see Figure 2).

\begin{figure}[ht]\label{fig2}
$$
0100\mathbf{000}\,0100\mathbf{100}\,0100\mathbf{000}\,0100\!*\!*\!*0100\!*\!*\!*0100\!*\!*\!*0100\mathbf{000}\,0100\mathbf{100}\,0100\mathbf{000}\,0100...
$$
\caption{\small On this figure $B_2=000100000$, $r_2=3$ and $q_2=6$.}
\end{figure}

Now the unfilled positions come in clusters of $q_2-r_2$ blocks of length $q_1-r_1$.
\medskip

We continue in this manner: in step $k+1$ we use a block $B_{k+1}$ consisting of all zeros but one symbol 1,
whose length equals $r_{k+1}$ (a freely chosen number) times $Q_k=(p_1-r_1)(p_2-r_2)\dots(p_k-r_k)$ (the number of unfilled positions in $[0,p_k-1]$ in the so far constructed sequence), we use this block to fill all unfilled places in $[0,r_{k+1}p_k-1]$, then we repeat it with a period $p_{k+1}=q_kp_k$ for some freely chosen $q_k>r_k$. Two more details must be taken care of: the products $\prod_{k=1}^N(1-\frac{r_k}{q_k})$, representing the density of unfilled positions after step $N$, must converge to a number $d>0$. The second requirement is that for each $k$ the symbol 1 appears in the future blocks $B_{k'}$ $(k'>k)$ at positions whose remainders modulo $Q_k$ assume every possible value infinitely many times. With such an arrangement it is not very hard to see that:
\begin{enumerate}
	\item the generated (bilateral) Toeplitz subshift \xt\ is an almost 1-1 extension of the odometer
	\gt\ with scale $H=(p_k)$;
	\item \xt\ is irregular and satisfies the condition SAR with the space $Y$ of aperiodic readouts
	consisting of all $\{0,1\}$-valued \sq s having at most one symbol 1.
\end{enumerate}
Clearly, $Y$ supports only one \im\ $\nu$, the pointmass at the fixpoint $(\dots000\dots)$. Theorem
\ref{tw210} (1)--(3) implies that \xt\ is uniquely ergodic, the unique \im\ $\mu$ is isomorphic to the only measure with marginals $\lambda$ and $\nu$, which is $\lambda\times\nu$. Since $\nu$ is supported by
one point, the factor map $\pi$ provides an isomorphism between $\mu$ and $\lambda$, as required. $\blacksquare$
\medskip

There exist strictly ergodic systems (also subshifts) isomorphic to an odometer \gt, yet whose maximal equicontinuous factor $(G',\tau')$ is a proper factor of \gt\ and Theorem~\ref{4.1} does not apply to such systems. For $G$ which is not simple\footnote{An odometer is \emph{simple} when its scale is $(p^k)$ for a prime number $p$. Simple odometers have no infinite proper factors, other do.}, such examples are easily obtained with $G'$ being an adding machine and the system is an almost 1-1 extension of $(G',\tau')$, while the remaining eigenvalues of $G$ are realized by discontinuous eigenfunctions (see \cite{D-L}). There exist also models for which the maximal equicontinuous factor is trivial (hence all eigenvalues of \gt\ are realized by discontinuous eigenfunctions and the system is \tl ly weakly mixing). Such examples can be produced for all odometers, including the simple ones. Below we give an example with even stronger property of \tl\ mixing.\footnote{It is known (\cite{Leh})
  that every aperiodic ergodic system has a topologically mixing strictly ergodic model. Here we provide a particular example.}

\begin{exam}\label{exe1}
Given an odometer \gt, there exists a strictly ergodic topologically mixing subshift \xt\
isomorphic to \gt.
\end{exam}

\noindent
\emph{Sketch of the construction.} Since we will be dealing with subshifts, $T$ will always denote the shift transformation, regardless of the domain. Let $H=(p_k)$ denote the scale of the odometer. Let $(X_0,T)$ be a regular Toeplitz subshift with maximal equicontinuous factor \gt, and let $x_0\in X_0$ be a
Toeplitz \sq. We will produce a \sq\ of \tl\ conjugacies of $(X_0,T)$, converging almost everywhere
to an isomorphism with the desired subshift \xt.

In step 1, choose some $r_1\in\na$ and find all periodic repetitions in $x_0$ of the central block $C_1=x_0[-r_1,r_1]$. The period of the repetitions is some $p_1\in H$. Choose $q_1$ such that $p_1q_1\in H$ and choose every $q_1$th periodic occurrence of $C_1$ in $x_0$ (avoiding the central one). Let us call these places \emph{1-windows}. Now comes the modification: within each 1-window we shift the contents one position to the left (sending the leftmost symbol to the right end). The modification passes over, in an obvious way, to all elements of $X_0$ and is invertible if $q_1$ is large enough\footnote{$q_1$ must be long enough so that the $p_1$-periodic part of any $x\in X_0$ can be determined by viewing any block of length $p_1q_1-2r_1-1$.}. This is our conjugacy $\Phi_1:X_0\to X_1$ between two Toeplitz subshifts. We
denote by $x_1$ the Toeplitz \sq\ $\Phi_1(x_0)$.

In step 2 we choose some $r_2$ and we denote $C_2=x_1[-r_2,r_2]$. We must take care of two details:
$C_2$ must be long enough to include several 1-windows, moreover, its ends must fall \emph{far} from
the 1-windows, for instance, approximately in the middle between two of them. In the future this will prevent an accumulation of the end-irregularities. We find all periodic occurrences of $C_2$ in $x_1$, and their period $p_2\in H$. Next we choose some large $q_2$ such that $p_2q_2\in H$ and we mark every $q_2$th
copy of $C_2$ in $x_2$ (avoiding the central one) as 2-windows. Like before, we shift the contents of each 2-window one position to the left, sending the leftmost symbol to the right end. This modification spreads naturally to a conjugacy $\Phi_2:X_1\to X_2$ between Toeplitz subshifts. We let $x_2=\Phi_2(x_1)=\Phi_2\Phi_1(x_0)$.
\smallskip

We proceed in this manner infinitely many times, assuring that the densities of the positions affected
by consecutive modifications (i.e., the ratios $\rho_k=\frac{2r_k+1}{p_kq_k}$) are summable, and that
the ends of the $k$-windows fall, for every $k'<k$ approximately in the middle between
the a pair $k'$-windows (see Figure 3).

\begin{figure}[ht]\label{fig3}
\begin{center}
\includegraphics[width=12cm]{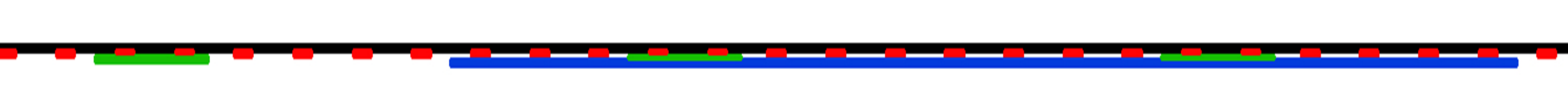}
\caption{\small The colors show the $k$-windows (i.e., areas affected by $\Phi_k$): original -- black, 1-windows -- red, 2-windows -- green, 3-windows -- blue (the next 3-window too far to be shown).
The shift by one position of red within green, and that of green and red within blue are to small
to be seen.}
\end{center}\end{figure}

Let $x_\infty$ denote the sequence obtained as the limit of $x_k$ (which exists because every position in $x_0$ is affected by the modifications at most finitely many times), and we let $X_\infty$ be the shift orbit closure of $x_\infty$ and finally $X$, a minimal subset of $X_\infty$.\footnote{This step allows us to avoid verifying whether $X_\infty$ is minimal.}
Notice that the maps $\Psi_k=\Phi_k\dots\Phi_2\Phi_1$ converge at every point of $X_0$ whose every coordinate is affected at most finitely many times. We let $\Psi$ be the (partially defined) limit map on $X_0$.

The following observations are standard and we skip their proofs:
\begin{enumerate}
	\item $(X_\infty,T)$ is \tl ly mixing (can be checked directly by the definition).
	\item $\Psi$ is defined almost everywhere on $X_0$ and acts into $X_\infty$.
	\item For each $k$, on $X$ we can define the following map $\psi_k$: in each element of $X$
	we can easily identify and reverse the consequences of the modification $\Phi_1$, $\Phi_2,\Phi_3,
	\dots,\Phi_k$ (although we ``invert'' $\Phi_k$ in a seemingly wrong order, these inverses commute).
	{\it Attention}: the maps $\psi_k$ are not precisely inverses of $\Psi_k$; they are defined on
	$X_\infty$,	not on 	$X_k$.
	\item The limit of the above maps $\psi_k$ (call it $\psi$) is defined almost everywhere on $X_\infty$
	for every \im.
	\item The range of $\psi$ is contained in $X_0$ and $\psi$ inverts $\Psi$ wherever the latter is defined.
\end{enumerate}
The last four facts imply that $(X_\infty,T)$ is uniquely ergodic and isomorphic to $(X_0,T)$. Since the minimal subset $X\subset X_\infty$ supports an \im, it supports the unique one, hence \xt\ is isomorphic to $(X_0,T)$ and, additionally, minimal. Thus \xt\ is the desired example. $\blacksquare$
\medskip

Although we have no general proof of Sarnak's conjecture for all models of odometers, interestingly,
it does hold in the above \tl ly mixing example. Again, we will only sketch the argument.

\medskip\noindent
\emph{Sketch of proof of Sarnak's conjecture in Example \ref{exe1}.} Since every continuous function on $X$ can be approximated in $L^1(\mu)$ by a (continuous) linear combination of characteristic functions of cylinders corresponding to finite blocks, it suffices
to verify the conjecture for such characteristic functions. We will do that for blocks of length
1, i.e., for the occurrences of the symbol 1. The argument for longer blocks is identical.

Fix an $\epsilon>0$ and let $k$ be such that the density of places where $x_\infty$ differs from $x_k$
(i.e., $\sum_{k'=k+1}^\infty\rho_{k'}$) is less than $\epsilon$. Since $x_k$ is regular Toeplitz, there are $l$ and $N$ such that the $p_l$-periodic part occupies a fraction at least $1-\epsilon$ in every block of length $N$ appearing in any element of $X_k$. The problem we are facing is that the (slightly perturbed) $p_l$-periodic part in $x\in X$ does not come from a continuous function on $X$ any more, and we have no guarantee that \emph{all} points realize the ergodic theorem for it. We need more subtle observations.

We can assume that $N$ is large enough so that
$$
\left|\sum_{i=1}^n\xi(i)\mmu(i)\right|<n\epsilon
$$
for every $\{0,1\}$-valued $p_l$-periodic \sq\ $\xi$ and every $n\ge N$. Let $x\in X$
and consider the initial block $x[1,n]$ for $n\ge \frac N\epsilon$ (this is an
arbitrary block $B$ of length $n$ appearing in $x_\infty$). We will argue that there
are (at most) three blocks $B_1,B_2,B_3$ appearing in $X_k$, with lengths summing to $n$,
such that $x[1,n]$ nearly equals the concatenation $B_1B_2B_3$ in the sense that
the fraction of disagreements is at most $2\epsilon$. Once this is proved, we can
write $B_1=x'[1,m], B_2=x''[m+1,m'], B_3=x'''[m'+1,n]$ for some $x',x'',x'''\in X_k$
and $m,m'\in[1,n]$, $m\le m'$, and then have the following estimation
\begin{gather*}
\left|\sum_{i=1}^n x(i)\mmu(i)\right| \overset{2n\epsilon}\approx
\left|\sum_{i=1}^m x'(i)\mmu(i)+\sum_{i=m+1}^{m'} x''(i)\mmu(i)+\sum_{i=m'+1}^{n} x'''(i)\mmu(i)\right|=\\
\left|\sum_{i=1}^m x'(i)\mmu(i)+\sum_{i=1}^{m'} x''(i)\mmu(i)-
\sum_{i=1}^{m} x''(i)\mmu(i)+\sum_{i=1}^{n} x'''(i)\mmu(i)-
\sum_{i=1}^{m'} x'''(i)\mmu(i)\right|<10n\epsilon,
\end{gather*}
because
\begin{itemize}
	\item the sums shorter than $N$ contribute at most $sN<sn\epsilon$ ($s\le 4$ is the number of such sums),
	\item in each sum not shorter than $N$
	\begin{itemize}
		\item$ x',x''$ or $x'''$ can be replaced by a $p_l$-periodic \sq\ and this will change the sum
		by less than $\epsilon$ times the summing length,
		\item once the above replacement is done, the absolute value of the sum does not exceed $\epsilon$
		times the summing length,
	\end{itemize}
	\item the sum of the summing lengths equals $(5-s)n$.
\end{itemize}

Clearly, this estimation ends the proof.

So, it remains to break a block $B$ of length $n$ appearing in $x_\infty$ into at most three subblocks,
as desired. We can think of $B$ as of a block $B_0$ appearing somewhere in $x_k$, affected by some finitely many modifications $\Phi_{k+1},\dots,\Phi_K$. If $B$ is entirely contained in a $K$-window without its endpoints, then $\Phi_K$ shifts the entire contents one position to the left, so that the result is the same as if we started from the block $B_0'$ lying in $x_k$ one position to the left with respect to $B_0$ and pretended that $\Phi_K$ did not affect it.
In this manner, we can ignore all such cases and move on to the largest index (and denote this one by $K$) for which only part of $B$ is affected by $\Phi_K$, that is, $B$ contains an endpoint of a $K$-window.
Now there are three possibilities:
\begin{enumerate}
	\item $B$ intersects two or more $K$-windows,
	\item $B$ contains one or both endpoints of just one $K$-window.
\end{enumerate}
In case (1) the fraction of modifications introduced by $\Phi_K$ in $B$ is at most $2\rho_K$
(and $\rho_{k'}$ for earlier modifications with $k'=k+1,\dots, K-1$), so $B$ differs from $B_0$
on a fraction of at most $2\epsilon$ places and there is no need to partition it (we put $B_1=B_0$, there is no $B_2$ or $B_3$). In case (2) we cut $B$ at the endpoints of the $K$-window. This produces two or three subblocks $B_1',B_2',B_3'$. Notice that $\Phi_K$ affects only one of these subblocks and only by shifting it ``in one piece''. So, we only need to see how much each subblock is affected by the earlier modifications. Recall that for each $k'<K$ the endpoints of the $k'$-windows fall approximately $\frac{p_{k'}q_{k'}}2$ places away from the endpoints of the $K$-window. This implies that if a modification $\Phi_{k'}$ does affect a subblock $B_i'$ ($i=1,2$ or $3$), then the fraction of
the modifications in this subblock is at most $2\rho_{k'}$. So each $B_i'$ differs from a subblock of
$x_k$ on a fraction of at most $2\epsilon$ places. This completes the argument. \qed
\medskip

%\begin{rem}
%The above proof of the validity of Sarnak's conjecture may fail if the distances between the endpoints
%of $k$-windows and $k'$-windows are smaller. In other words, one can modify the example so it remains
%a \tl ly mixing model of the odometer, and for which the above calculation does not work.
%At the moment we have no universal argument for such models. !!!SPRAWDZI\'C CZY MOZE TO JEST NIEPOTRZEBNE!!!
%\end{rem}

\section{Toeplitz \sq s which fail Sarnak's conjecture}

We will say that \sq s $\xi$ and $\eta$ are \emph{weakly} (resp. \emph{strongly}) correlated if the upper (resp. lower) limit of $\frac 1n\sum_{i=1}^n\xi(i)\overline\eta(i)$ is positive.

\subsection{Weak failure}

\begin{exam}\label{exem}
There exists an irregular $\{-1,1\}$-valued Toeplitz subshift $X$ such that some $x\in X$ are strongly correlated with $\mmu$. The set of such points $x$ is dense in $X$. Moreover, weak correlation holds on a residual subset of $X$.
\end{exam}

\noindent
The example is very simple, once the general construction of Toeplitz systems is understood: Consider
an irregular Toeplitz subshift satisfying the condition SAR and such that $Y$, (the space of aperiodic readouts) is the full shift on two symbols $\{-1,1\}$ (we skip the detailed construction of such a subshift;
it is done by a ``standard method'' e.g., the Oxtoby technique, see \cite{survey}). Let $\pi:X\to G$ denote the maximal equicontinuous factor map onto the underlying adding machine and let $\lambda$ be the Haar measure on $G$. Then, by Theorem \ref{irreg} (and the explanations following that theorem),
there is a set $G'\subset G$ with $\lambda(G')=1$ which satisfy the following two conditions:

\begin{enumerate}
	\item all points $x$ in the fiber $\pi^{-1}(g)$ agree along a common periodic part whose density equals $1-d$,
	\item as $x$ ranges over $\pi^{-1}(g)$, all possible $\{-1,1\}$-valued \sq s occur along the aperiodic part of $x$.
\end{enumerate}
In particular, if we arrange that $d>1-\frac3{\pi^2}$ ($\approx 0.7$) (which is easily done within the ``standard method'') then,
for every $g\in G'$ and $x\in\pi^{-1}(g)$ the set $\mathsf A = \Aper(x)\cap\{n\ge 0: \mmu(n)\neq 0\}$ has positive lower forward density at least $d_0=d+\frac6{\pi^2}-1> \frac3{\pi^2}$. There exists a point $x_0\in \pi^{-1}(g)$ such that for $n\in \mathsf A$ it equals $\mmu(n)$. It is obvious that (even in the ``worst case scenario'', when $x(i)=-\mmu(i)$ and $|\mmu(i)|=1$ everywhere on the periodic part of $x$) we still have
\begin{equation}\label{ll}
\liminf_n\frac 1n\sum_{i=1}^nx_0(i)\mmu(i) \ge d_0 - (1-d)>0.
\end{equation}

Next, we will show that points $x_0$ as constructed above (satisfying \eqref{ll}) lie densely in $X$.
Consider a basic open set $U$ in $X$, i.e., a cylinder corresponding to a block $B\in\{-1,1\}^{2k+1}$ occurring in $X$ at the coordinates $[-k,k]$.
By minimality, the same block occurs (perhaps at a different place) in the generating Toeplitz
\sq, which implies that the same block occurs somewhere in the periodic part of every element
of $X$, in particular in an element $x\in\pi^{-1}(G')$. Notice (directly from the definition)
that the set $G'$ is invariant (equivalently, $\pi^{-1}(G')$ is shift-invariant). Thus, by an appropriate shifting, we obtain a new point $x$ such that $B$ occurs in $x$ at the coordinates $[-k,k]$
and still belongs to the periodic part of $x$, and $g=\pi^{-1}(x)$ belongs to $G'$. Using this
particular $g$ in the above construction of $x_0$ we produce the point $x_0$ such that $x_0[-k,k]=B$
($x_0$ belongs to the same fiber as $x$ and thus agrees with $x$ along the periodic part, which
includes the coordinates $[-k,k]$). In other words, we have constructed a point $x_0\in U$ satisfying \eqref{ll}.

Next we observe that if we weaken \eqref{ll} by requiring that the \emph{upper} limit is larger than or equal to
a positive $\epsilon < d_0-(1-d)$, then it holds on a residual set. Indeed, we can write
\begin{multline*}
\left\{x: \limsup_n\frac 1n\sum_{i=1}^nx(i)\mmu(i) > \epsilon\right\} \subset \\ \bigcap_{m\ge 1}\bigcup_{n\ge m}\left\{x:\frac 1n\sum_{i=1}^nx(i)\mmu(i) > \epsilon\right\}\subset \\ \left\{x: \limsup_n\frac 1n\sum_{i=1}^nx(i)\mmu(i) \ge \epsilon\right\}.
\end{multline*}
The first set contains the dense set of points satisfying $\eqref{ll}$, the middle set is of type $G_\delta$ (hence it is a dense $G_\delta$) and thus the last set is residual. The example is completed.
$\blacksquare$
\medskip

Since Toeplitz \sq s form a residual subset inside a Toeplitz subshift, we conclude that

\begin{cor}\label{bad Toeplitz}
There exist (irregular) Toeplitz \sq s weakly correlated with the M\"obius function.
\end{cor}

We remark, that the Toeplitz subshift of the above example has positive entropy (equal to $d\ln2$), hence it stands in no collision with the Sarnak's conjecture.

\subsection{Strong failure}
The following example is replicated from \cite{Lem}.

\begin{exam}\label{exlem}
There exists a one-sided Toeplitz \sq\ strongly correlated with the M\"obius function.
\end{exam}

\noindent
We begin by describing a general scheme (used in \cite{Lem}) of producing a one-sided Toeplitz \sq\ from another symbolic \sq.
Let $y=(\text{y}_n)_{n\ge 1}$ be a one-sided \sq\ over a finite alphabet $\Lambda$. Let $H=(p_k)$ be a scale of an adding machine such that $p_1\ge 3$. Since $p_{k+1}\ge 2p_k$ for every $k$, this condition implies that $\rho=\sum_{k\ge 1}\frac1{p_k}<1$. We define the associated one-sided Toeplitz \sq\ $x$ as follows:
$$
x_y(n) =
\begin{cases} \text{y}_1,& n=1\mod p_1\\
\text{y}_2,& n=2\mod p_2\\
\vdots\\
\text{y}_{p_1},& n=p_1\mod p_{p_1}\\
\text{y}_{p_1+1},& n=p_1+2\mod p_{p_1+1}\\
\text{y}_{p_1+2},& n=p_1+3\mod p_{p_1+2}\\
\vdots\\
\text{y}_{2p_1-1},& n=2p_1\mod p_{2p_1-1}\\
\text{y}_{2p_1},& n=2p_1+3\mod p_{2p_1}\\
\vdots
\end{cases}
$$
We refrain from further detailed listing, as it becomes too complicated. The simple rule behind the scheme
is that $\text{y}_k$ is placed at the first position available after steps $1,2,\dots,k-1$ and then it is repeated periodically with the period $p_k$. This concludes the description of the scheme.
\medskip

For further considerations, it will be convenient to highlight, for each $k$, the first placement of $\text{y}_k$ in $x_y$ (it is shown in boldface, while its further periodic repetitions are printed in the normal font). While reading the following text and diagrams it is important to distinguish between boldface symbols $\mathbf{y}_k$ and normal font symbols $\text{y}_k$. The diagram below shows the filling scheme in case $p_k=3^k$ with the boldface terms marked.
$$_{
\mathbf{y_1}\mathbf{y_2}\mathbf{y_3}\text{y}_1\mathbf{y_4}\mathbf{y_5}\text{y}_1\mathbf{y_6}\mathbf{y_7}\text{y}_1\text{y}_2\mathbf{y_8}\text{y}_1\mathbf{y_9}\mathbf{y_{10}}\text{y}_1\mathbf{y_{11}}\mathbf{y_{12}}\text{y}_1\text{y}_2\mathbf{y_{13}}\text{y}_1\mathbf{y_{14}}\mathbf{y_{15}}\text{y}_1\mathbf{y_{16}}\mathbf{y_{17}}\text{y}_1\text{y}_2\text{y}_3\text{y}_1\mathbf{y_{18}}\mathbf{y_{19}}\text{y}_1\mathbf{y_{20}}\dots}
$$

\medskip\noindent
The authors of \cite{Lem} show that the lower density of the boldface symbols (which they call \emph{initials}) is at least $1-\rho$ (combining this with Lemma \ref{shiftz} (3) below we see that in fact these symbols have density $1-\rho$), which can be made arbitrarily close to 1. They select $y$ so that $x_y(n) = \mmu(n)$ whenever $x_y(n)$ is a boldface symbol. If $1-\rho> 1-\frac3{\pi^2}$ then, for the same reasons as in \eqref{ll} (with $1-\rho$ in the role of $d$), they obtain that $x$ is strongly correlated with $\mmu$. Of course, in view of Sarnak's conjecture, one is obliged to compute the \tl\ entropy of the generated Toeplitz subshift (at least to check whether it is positive). We will do so in Section \ref{entropy}. $\blacksquare$

\section{Properties of the scheme}

We are interested in properties of Toeplitz \sq s obtained through the above scheme for general \sq s $y$. In particular, we would like to know whether positive entropy follows automatically from positive entropy of the orbit closure $Y$ of~$y$. As we soon show, the answer is negative. This is quite unfortunate, because it forces us to estimate the entropy of the example of \cite{Lem} using tedious methods adapted the the particular example.
\medskip

So, consider a general \sq\ $y$ and the associated Toeplitz \sq\ $x_y$ with the boldface symbols marked. The following lemma addresses the distribution of the boldface symbols in $x_y$. The statement (5) will be used immediately in Example \ref{7.2}, while statement (4) only in Section \ref{entropy}. Statements (1) and (2) are necessary to prove (4), while (3) is just a digression noted in passing.

\begin{lem}\label{shiftz}
Let $z$ be the $\{0,1\}$-valued \sq\ given by the rule $z(n)=1\iff x_y(n)$ is a boldface symbol.
Then
\begin{enumerate}
	\item The frequency of zeros in the block $B_{k,0}=z[1,p_k]$ converges to $\rho$ from below, as 		
	$k\to\infty$.
	\item For every $k$ every block of the form $B_{k,j}=B_{k,j}[1,p_k]=z[jp_k+1,(j+1)p_k]$ ($n\ge 0$)
	can be obtained from $B_{k,0}$ by only replacing some $1$'s by $0$'s.
	\item The symbols $0$ in $z$ have lower Banach density $\rho$.
	\item Given $\epsilon>0$ there is a $\delta>0$ and $n_0\in\na$, such that, for any $n\ge n_0$, the 		
	cardinality of different blocks $B$ of length $n$, appearing in $z$ and in which the frequency of
	$0$'s is at most $\rho+\delta$, does not exceed $2^{n\epsilon}$.
	\item For every natural $m$, $z$ contains a block consisting of $m$ single symbols $1$ separated
	by blocks of zeros of lengths at least $m$. In particular, the upper Banach density of zeros in $z$ is
	$1$.
\end{enumerate}
\end{lem}

\begin{rem}
(3) and (5) imply that the subshift generated by $z$ is not uniquely ergodic; at least one \im\
assigns to the cylinder of $0$ the value $\rho$ and at least one -- the value $1$ (perhaps there are more possibilities). (4) implies that every measure of the first kind has entropy zero. We have not verified whether this subshift has \tl\ entropy zero (regardless of the scale $H=(p_k)$).
\end{rem}

\begin{proof} For (1) it suffices to observe that the frequency of zeros in $B_{k,0}$ equals
$$
\tfrac1{p_k}\sum_{k'<k}(\tfrac{p_k}{p_{k'}}-1)=\Bigl(\sum_{k'<k}\tfrac1{p_k'}\Bigr)-\tfrac {k-1}{p_k}.
$$
Indeed, in step $k'\le k$, in $x_y[1,p_k]$ we have placed $\frac{p_k}{p_{k'}}$ symbols, of which one was boldface. Hence the formula.

\smallskip
For (2) note that $B_{k,0}(i)=0$ if and only if $x_y(i)$ is a normal font (i.e., repeated)
symbol $\text{y}_{k'}$ for some $k'$. This is possible only when $p_{k'}<p_k$. But then $p_k$ is a
multiple of $p_{k'}$ which implies that $B_{k,j}(i)=0$ for every $j$.
\smallskip

Clearly, (1) implies that the lower density (and lower Banach density) of zeros is at most $\rho$. By (2), the frequency of $0$'s in any $B_{k,j}$ may only be larger than that in $B_{k,0}$ (which is close to $\rho$). Every sufficiently long block $B$ in $z$ is a concatenation of the blocks $B_{k,j}$ and negligibly small prefix and suffix, so the frequency of $0$'s in $B$ is not less than $\rho$ minus a negligibly small error term. This proves (3).
\smallskip

For (4) we argue as above: every block $B$ of large length $n$, after removing negligibly small prefix and suffix, becomes a concatenation of the blocks $B_{k,j}$ (with a large parameter $k$). By (2), this concatenation can be viewed as a periodic repetition of $B_{k,0}$ with some $1$'s replaced by $0$'s.
But the number of replaced symbols $1$ cannot essentially exceed $n\delta$, otherwise the overall frequency of $0$'s would be too large. Such a replacement can be performed in approximately $e^{n(-\delta\ln\delta-(1-\delta)\ln(1-\delta))}$ different ways, which, for an appropriately small $\delta$, is smaller than $2^{n\epsilon}$.
\smallskip

For (5), we will need the following (somewhat lengthy)

\medskip
\noindent
{\bf Claim}. For every $m\ge 1$, after some number $k_m$ of steps of filling $x_y$
(i.e., after having placed the periodic repetitions of $\text{y}_1,\dots,\text{y}_{k_m}$), $x_y$ starts with the following configuration (later referred to as $C$): a continuous entirely filled block (with both boldface and normal font symbols) ending with the boldface $\mathbf{y}_{k_m}$ followed by a single unfilled coordinate, next a continuous block filled with normal font terms followed by a single unfilled position, next again a continuous block filled with normal font terms followed by a single unfilled position, and so on. The continuous filled blocks (including the first one) have strictly decreasing lengths and there is $m$ of them (see the diagram below for $m=6$, most of the indices are omitted). We do not require that the last unfilled position is single (it may be followed by more unfilled positions).
$$
\mathbf{y\!}_1\mathbf{y}\,\mathbf{y}\,\text{y}\,\mathbf{y}\,\mathbf{y}\,\text{y}\,\mathbf{y\!}_{k_m}\boxed{\vphantom{a}}\,\text{y}\,\text{y}\,\text{y}\,\text{y}\,\text{y}\,\text{y}\boxed{\vphantom{a}}\,\text{y}\,\text{y}\,\text{y}\,\text{y}\,\text{y}\boxed{\vphantom{a}}\,\text{y}\,\text{y}\,\text{y}\,\text{y}\boxed{\vphantom{a}}\,\text{y}\,\text{y}\,\text{y}\boxed{\vphantom{a}}\,\text{y}\,\text{y}\boxed{\vphantom{a}}
$$

\medskip
\noindent{\it Proof of the Claim.} For $m=1$ the condition is fulfilled after $k_1=1$ steps, so the induction starts. Suppose the claim holds for some $m\ge 1$. The pattern $C$ is repeated periodically with the period $p_{k_m}$ and the repetitions cannot overlap (because the lengths of the filled blocks are all different). This implies that the pattern $C$ is contained in $x_y[1,p_{k_m}]$. Let us move to the first repetition of the pattern $C$ further to the right (call it $C'$). It starts at the position $p_{k_m}+1$ and clearly, here all symbols are printed in normal font. Notice that at least two preceding positions: $p_{k_m}$ and $p_{k_m}-1$ are not occupied (because $p_{k_m}=0\mod p_k$ and $p_{k_m}-1=p_k-1\mod p_k$ for any $k\le k_m$, while the positions filled with the symbol $\text{y}_k$ have values $\mod$ $p_k$ positive and much smaller than $p_k$). Now we perform the construction steps $k_m+1$, $k_m+2$, etc., of filling in $x_y$, until we fill the position $p_{k_m}-1$ (with a boldface symbol $\mathbf{y\!}_{k'}$ for some $k'>k_m$). Notice that the pattern $C'$ reaches to at most the position $2p_{k_m}$, which is smaller than $p_{k_m+1}+1$, so the repeated (normal font) symbols added in these new steps fall to the right of $C'$ (i.e., they do not affect it). In this manner we fill all the unfilled positions within $x_y[1,p_{k_m}-1]$ creating (together with $C'$) a pattern as required for $m+1$ in the induction (with $k_{m+1}=k'$). Note that the new initial continuously filled block $x_y[1,p_{k_m}-1]$ has length $p_{k_m}-1$ larger than or equal to the length of the pattern $C'$ (perhaps without counting its last empty cell), in particular, for $m>1$, it is strictly longer than the first completely filled block of $C'$. For $m=1$ this also holds, because, in this case, the length of the first (and unique) filled block of $C'$ is 1, while $p_1-1>1$. The claim is thus proved.
\medskip

In the following construction steps, the unfilled positions in the pattern $C$ are filled with the boldface symbols $\mathbf{y}_{k_m+1}, \mathbf{y}_{k_m+2},\dots,\mathbf{y}_{k_m+m}$. Because the lengths of the separating normal font blocks strictly decrease, and there is $m-1$ of them, the first one has length at least $m-1$, the next one $m-2$ and so on. Thus, taking for simplicity $m$ to be even, we obtain that each of the boldface terms $\mathbf{y}_{k_m+1},\mathbf{y}_{k_m+2},\dots,\mathbf{y}_{k_m+\frac m2}$ is followed (and preceded) by a block of normal font symbols of length at least $\frac m2$. Renaming $\frac m2$ as $m$ ends the proof.
\end{proof}

It follows from the construction (or we can easily arrange it by choosing a sub\sq) that $k_{m+1}>k_m+m$ for each $m\ge 1$.

\begin{exam}\label{7.2}
There exists a one-sided symbolic \sq\ $y$ such that its orbit closure $Y$ has positive \tl\ entropy and supports many \im s, yet the associated Toeplitz subshift $X_y$ (the orbit closure of $x_y$) is strictly ergodic with entropy zero.
\end{exam}

\noindent
We just need to decide about the contents of the \sq\ $y$. Let $\Lambda=\{0,1\}$. For each $m\ge 1$ we  let $y[k_m+1,k_m+m]$ be a block $A_m$ and $y[k_m+m+1,k_{m+1}]$ be the block consisting entirely of zeros. We arrange that the \sq\ of blocks $(A_m)_{m\ge 1}$ generates a positive entropy subshift $Y_0$ with many \im s (for example, the full shift on two symbols). It is clear that the orbit-closure $Y$ of $y$ contains $Y_0$, hence has positive entropy and many \im s.
\medskip

Let us ignore ``accidental'' periodic repetitions of symbols in $x_y$. This is to say, we will denote
by $\text{Per}_k(x_y)$ the set of positions of the $p_{k'}$-periodic repetitions of the symbols $\text{y}_{k'}$ for $k'\le k$. The density of so defined $\text{Per}_k(x_y)$ is $\sum_{k'=1}^k\frac1{p_{k'}}$.

It follows from the general facts concerning Toeplitz subshifts, that if $\mu$ is an \im\ on $X_y$, then
$\mu$-almost every $x\in X_y$ has the ``non-accidental'' periodic part of density $\rho<1$, and the remaining part (which we denote by $\text{Aper}(x)$, although at the moment we only know it {\it contains} the true aperiodic part). Clearly, $\text{Aper}(x)$ is infinite as it has density $1-\rho$.

Suppose $\text{Aper}(x)$ contains two positions $n$ and $n+m$ such that $x(n)=x(n+m)=1$.
This implies that in $x_y$ there are infinitely many positions $n'$ such that $x(n')=x(n'+m)=1$
and $n'$ and $n'+m$ both belong to arbitrarily high periodic parts. This is to say, $x_y(n')=\text{y}_{k'}$ (or $\mathbf{y}_{k'}$) and $x_y(n'+m)=\text{y}_{k''}$ (or $\mathbf{y}_{k''}$), where $k',k''$ are arbitrarily large, for instance larger than both $k$ and $k_m$ and such that $p_{k'},p_{k''}>m$.
Suppose $k'<k''$ (the other case is symmetric). Shift (if necessary) the window $[n',n'+m]$ to the
left by a multiple of $p_{k'}$ so it
starts with the boldface symbol $\mathbf{y}_{k'}$. Say, this is now $[n'',n''+m]$. The position $n''+m$
cannot be occupied by $\text{y}_{k'''}$ with $k'''\le k'$ because then $n'+m$ would also be occupied by
the same $\text{y}_{k'''}$ (while it is by $\text{y}_{k''}$ or $\mathbf{y}_{k''}$ with $k''>k'$). This implies that $n''+m$ is occupied by some $\text{y}_{k'''}$ or $\mathbf{y}_{k'''}$ with $k'''>k'$.
But in such case, since $m$ is smaller than $p_{k'''}$, it must be the first occurrence, i.e., $\mathbf{y}_{k'''}$. Since $\mathbf{y}_{k'}=1$, $k'$ must belong to an interval $[k_{m'},k_{m'}+m']$ for some $m'$ and since $k'>k_m$, $m'$ must be larger than or equal to $m$. This implies that the occurrence of $\mathbf{y}_{k'}$ in $x_y$ (it occurs as $x_y(n'')$) is followed by a block of at least $m$ normal font symbols. This is a contradiction since we have just shown that $x_y(n''+m)$ is a boldface symbol.
\smallskip

We have proved that if $x$ has an infinite aperiodic part, this part is filled with zeros except perhaps
one 1. This immediately implies that $X_y$ has entropy zero and is strictly ergodic (in fact, it is isomorphic to the odometer, like the system of Example \ref{exe}).

\medskip

Although the above statement already captures the most important properties of the Toeplitz subshift $X_y$,
we have not yet guaranteed that
\begin{enumerate}
	\item the \sq\ $x_y$ is irregular with the density of a ``true'' aperiodic part equal to $1-\rho$,
	\item the odometer $(G,\tau)$ is a factor of $X_y$.
\end{enumerate}
All these features must be arranged separately, by delicate modifications of $y$, yet, which do not destroy what we have already achieved. We will only outline what needs to be done, skipping the tedious and not very interesting details.

We must realize that the construction steps of filling in $x_y$ corresponds to successively defining
the associated semicocycle as constant on some clopen subsets $C_k\subset G$. The positions of these subsets are determined by the scheme; they form a dense subset of $G$ and have jointly the Haar measure $\rho$.

For (1) we need to assure that the cocycle is \emph{discontinuous} at every point of the complementary set $D\subset G$. This can be done by making sure that we assign at least two different values in every neighborhood of every point of $D$, which can be achieved by modifying (if necessary) the values of $y$ along a very sparse sub\sq, so sparse that it would not affect other properties. Notice that every neighborhood of every point in $D$ contains infinitely many sets $C_k$, so we can choose an arbitrarily sparse subsequence of these sets which visits all such neighborhoods.

Likewise, for (2) we need the semicocycle to be invariant under no rotation. For this is suffices that
we arrange a discontinuity point ``unlike any other''. This can also be done by very sparse modifications of $y$. $\blacksquare$

\section{Entropy of the Example \ref{exlem}}\label{entropy}

We have eliminated the possibility of an ``automatic'' proof that the entropy in Example \ref{exlem}
is positive just based on the fact that $\mmu$ generates a subshift with positive \tl\ entropy.
On the other hand, in view of Sarnak's conjecture, we are obliged to check positivity of its topological entropy.

\begin{proof}[Proof of positivity of the entropy of Example \ref{exlem}]

For short, we will call the squares of prime numbers \emph{the p-squares}. They will be
denoted by $d_1,d_2, d_3,...$.
\medskip

Let $B$ denote an arbitrary block appearing in the subshift generated by $|\mmu|$. Let $n$ denote the
length of $B$, which we assume is large. Our goal is to indicate a place (an interval of $n$ consecutive coordinates) where $B$ occurs in $|\mmu|$ and in the Toeplitz \sq\ $x_y$ the number of boldface symbols is close to $n(1-\rho)$ (i.e., nearly realizes the upper Banach density of such symbols). Recall that $\rho$ is the sum of the inverses of the periods $p_k$ and is smaller than $\frac3{\pi^2}$. First we will argue
that finding such places (for all long enough blocks $B$) suffices for positivity of the \tl\ entropy of $X_y$.

Indeed, let $C$ denote the block appearing in $x_y$ over this interval. Then $B$ can be reconstructed knowing $C$ and two additional data: the positions of all the normal font symbols in $x_y$ in the considered interval and the contents of $|\mmu|$ at these positions. Since the number of normal font symbols is not larger than $n(\rho+\delta)$, Lemma \ref{shiftz} implies that there are at most $2^{n\epsilon}$ possibilities as to how the normal font symbols are distributed, and then there are
at most $2^{n(\rho+\delta)}$ possibilities as to their contents in $|\mmu|$. This produces the estimate
$$
\#\{B\} \le \#\{C\}\cdot 2^{n\epsilon}\cdot 2^{n(\rho+\delta)},
$$
where $\#\{B\}$ and $\#\{C\}$ denote the cardinalities of blocks of length $n$ in $|\mmu|$ and in $x_y$, respectively. Since $|\mmu|$ generates a subshift of entropy $\frac6{\pi^2}\log 2$, i.e., $\#\{B\}$ is nearly $2^{6n/\pi^2}$, the cardinality $\#\{C\}$ is (ignoring the small terms) nearly
$2^{n(6/\pi^2-\rho)}$, which yields positive \tl\ entropy of $X_y$ whenever $\rho<\frac6{\pi^2}$
(while we have assumed it is smaller even than half of that number).

\medskip
So, we focus on finding an interval, as specified at the start of the proof.
Find an interval $I$ of length $n$ where $B$ occurs in $|\mmu|$ (there is such). Positions of zeros in $I$ can be divided in two classes: first class -- coordinates divisible by any of the p-squares $d_1,\dots,d_L$, where $d_L$ is
the largest p-square smaller than $n$, and second class -- the remaining ones (which are divisible by
larger p-squares). Note that for each p-square larger than or equal to $n$ only one of its multiples can
occur in $I$.
%If there are no zeros of the second class, we artificially assign that the last zero is of both first and second class.
Now by the Chinese Reminder Theorem (see e.g. \cite{N},Chap. I)
we can shift the interval $I$ (and call the shifted interval $I'$), so that
\begin{enumerate}
	\item the shift is by a multiple of $d_1d_2\cdots d_K$, where $d_K$ is the largest p-square smaller than
	$4n^2$, and
  \item zeros of the second class become (after shifting) divisible by some a priori selected large
  p-squares $e_1,...,e_q$ (the choice of these p-squares will be specified in a moment).
\end{enumerate}
The zeros of the first class appearing in $|\mmu|$ over the interval $I'$ are precisely the shifted zeros of the first class over $I$. All zeros of the second class occurring over $I$ correspond (via the shift)
to zeros of the second class occurring over $I'$, but the later interval can have more zeros of the second class (some new zeros divisible by p-squares larger than $4n^2$ and different from $e_1,e_2,\dots, e_q$ can occur here). The configuration of zeros inherited from $I$ is repeated in $|\mmu|$ in every interval $I''$ along an arithmetic progression starting with $I'$ and advancing with step
$$
M= d_1\cdots d_K\cdot e_1\cdots e_q,
$$
(this need not be the smallest step, just one which is sure). From now on we will observe only the
intervals $I''$ appearing along this progression. In every such interval the additional zeros
(if there are any) must be divisible by p-squares larger than or equal to $4n^2$. Using Lemmas \ref{staszek1} and \ref{staszek2}, provided at the end of the paper, one easily obtains, that the percentage (among the observed intervals)
of intervals where there are any additional zeros does not exceed
$$
n\left(1-\prod_{j> K}(1-\tfrac1{d_j})\right)\le\frac n{\sqrt{d_{K+1}}}\le \frac12.
$$
In other words, in at least around 1/2 of the observed intervals in $|\mmu|$ there occurs precisely the block $B$.

\medskip
It now suffices to arrange that majority (a bit over 1/2 is enough) of these intervals are such that
in the Toeplitz \sq\ $x_y$ there are nearly $n(1-\rho)$ boldface symbols.

Recall that $H=(p_k)$ is the scale used to construct $x_y$, and for each $k\ge 1$, $p_{k+1}$ is an \emph{essential} multiple of $p_k$ (at least times 2). We are going to mark three ``important points''
$k_1,k_2,k_3$ on the axis of the parameter $k$ (these points bear hidden dependence on $n$, not
visible in the denotation).
\smallskip

\begin{itemize}
\item[$k_1$:]  Let $k_1$ be the largest $k$ such that $p_k<n$. Notice that $p_k\ge 2^k$ implies
$k_1<\ln n/\ln 2$, which (for large enough $n$) is smaller than $[n/\ln n]$.
\end{itemize}

\smallskip
Next ``important points'' require auxiliary functions. Let $N_k$ be the number of prime factors (in the meaning sum of their multiplicities) of the largest common divisor of $p_k$ and $d_1\cdots d_K$. The function $k\mapsto N_k$ is nondecreasing and becomes constant before it reaches
$2K+1$. This implies it eventually lies below the line $k/2$.

\begin{itemize}
\item[$k_2$:] Let $k_2$ be the largest $k$ (if such exists), for which $N_k\ge k/2$. Notice that
$N_k\le 2K$ implies $k_2\le 4K$, which (for $n$ large) is small compared to $n$ (equal to four times
the number of primes smaller than $2n$, i.e., approximately $8n/\ln 2n$). At this point we agree
that if $k_2<8n/\ln 2n$ (or does not exist at all), then we put $k_2 = [8n/\ln 2n]$. In particular,
this guarantees that $k_2>k_1$.
\end{itemize}

\begin{itemize}
\item[$k_3$:] Let $k_3$ be some place, not smaller than $k_2$, where the function $N_k$ has already reached its maximum. At this moment we define the numbers $e_1,...,e_q$, to be relatively prime with $p_{k_3+4q}$. We can do it now, because the particular values of $e_1,\dots, e_q$ have not been used in defining the preceding points.
\end{itemize}

%\begin{itemize}
%	\item[$k_4$:]
 Now consider a similar auxiliary function $M_k$ defined analogously as $N_k$ with $M$ in place of $d_1\cdots d_K$. Note that to the right of $k_3+4q$ the function $N_k$ does not grow, while $M_k$ may increase by at most $2q$ (possibly even in one jump), however, thank to the specific choice of $e_1,...,e_q$, this function will never again cross the line $k/2$.
%We let $k_4$ be some point where $M_k$ has already reached its maximum. %A bit later we will specify another lower bound
%for $k_4$.
%\end{itemize}

The ``important points'' are shown on the figure below
\begin{figure}[ht]\label{fig4}
\begin{center}
\includegraphics[width=12cm]{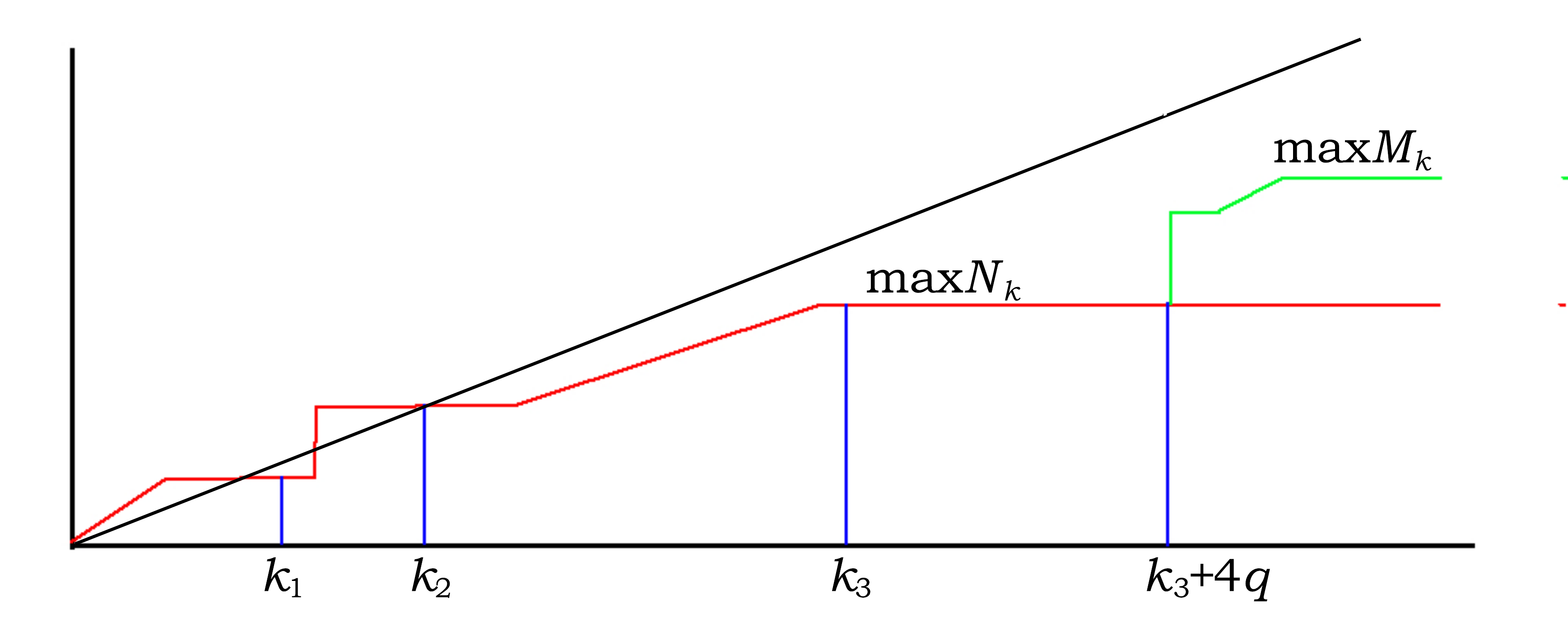}
\caption{\small The function $N_k$ is shown in red, and $M_k$ in green.}
\end{center}\end{figure}

The positions of normal font symbols in $x_y$ (equivalently, of zeros in $z$) we divide into arithmetic
progressions with periods $p_k$ (along such a progression $x_y(i) = \text{y}_k$) and we classify these
progressions in three groups.
\begin{enumerate}
	\item $k=1,\dots,k_1$
	\item $k=k_1+1,\dots,k_2$
\item $k>k_2$
%	\item $k=k_2+1,\dots,k_4$
%	\item $k>k_4$
\end{enumerate}
(The number $k_3$ is needed only to define $e_1,\dots,e_q$ and is not used to separate the groups.)
We will now analyze the progressions according to this classification.

\smallskip\noindent
1. Each progression from the first group occupies in every interval of length $n$ approximately $n/p_k$
positions (at most $2n/p_k$) for the largest $k$ in the group). The union of these progressions occupies not more than $n(\rho+\epsilon)$ (for large enough $n$).

\smallskip\noindent
2. Every progression from the second group is represented in every interval of length $n$ by at most
one coordinate. Jointly these groups occupy at most $4K$ positions, which does not essentially exceed $8n/\ln 2n$, i.e., negligibly little compared to $n$. \medskip

So far we have been estimating the number of normal font symbols in an interval of length $n$, and so far
it came out close to $n\rho$ (i.e., as we need it). From now on we will estimate the percentage of the
``observed'' intervals $I''$ which are disjoint from the progressions belonging to the third class.

3. Consider a $k$ in the third group. The period $p_k$ has with $M$ no more than $k/2$ common prime divisors. However, $p_k$ has at least $k$ prime factors, each equal to at least~$2$. This implies that $p_k/gcd(p_k,M)\ge 2^{k/2}$.
%This implies that
%after reducing the largest common divisor, of $p_k$ there remains a factor at least $2^{k/2}$ relatively prime with $M$.
%Hence that among the observed intervals $I''$ at most one in every $2^{k/2}$ intersects, at a specified position, the progression with step $p_k$.
Applying Lemma \ref{nowy} we conclude that  the percentage of intervals $I''$ intersecting the progressions from the third group does not exceed $n\sum_{k>k_2}2^{-k/2}$.
Since $k_2\ge 8n\log 2n$, this estimate is arbitrarily small for large $n$.

%\smallskip
%\textcolor{red}{WYDAJE MI SIE, ZE MOZNA ZASTOSOWAC ROZUMOWANIE Z 3. DO WSZYSTKICH $k>k_2$! PUNKT 4. STALBY SIE ZBEDNY.!!!}

%\noindent
%4. Finally, we address the fourth group as one subset (density is not countably additive).
%As \textcolor{red}{$k_4$   !!SK: BYLO $r_4$!!} increases (and we are still free to decide its size) the joint density of the normal font
%symbols from this group tends to zero. This follows easily from the fact that the normal font symbols
%have upper density equal to the limit of the densities of the $p_k$-periodic parts. Thus, the relative
%density in an arbitrary subset which has density (in this case the union of the intervals $I''$)
%also tends to zero. Hence, by choosing \textcolor{red}{$k_4$} large enough we can get that vast majority of the
%observed intervals $I''$ do not intersect the progressions from the last group.
%\medskip

To summarize, we can arrange that more than half of the observed intervals $I''$ do not intersect
any of the progressions from the third and fourth groups. In such intervals $|\mu|$ reads $B$.
Combining this with a previous estimate we obtain that there exist intervals in which both $|\mmu|$ reads $B$ and in $x_y$ there are nearly $n(1-\rho)$ boldface symbols, as required. This completes the proof.
\end{proof}

Now the missing lemmas.

\begin{lem}\label{nowy}
Consider a collection of arithmetic progressions of natural numbers $A_j=\{kp_j+r_j: k\in\na\}$, where $r_j\ge 0$ and $p_j/j\rightarrow \infty$ as $j\rightarrow \infty$. Let $M$ be a natural number and $r$ a nonnegative integer. The upper density of the set
$$
\{k\in\na:kM+r\in\bigcup_{j\ge 1}A_j\}
$$
is less than or equal to $\sum_{j\ge 1}\frac{1}{p_j'}$, where $p_j'=\frac{p_j}{gcd(p_j,M)}$ for $j\ge 1$.
\end{lem}

\begin{proof} Given $n$ let $J_n$ denote the maximal number $j$ such that $nM+r\ge p_j$.
Note that for any $n$ and $j$:
$$
|\{k\in\na: k\le n, kM+r\in A_j\}|\le\frac{n}{p'_j}+1
$$
and the set on the left hand side is empty if $j>J_n$.
It follows that
$$
\begin{array}{l}
\frac{1}{n}|\{k\in\na: k\le n, kM+r\in \bigcup_{j\ge 1}A_j\}|=\frac{1}{n}|\{k\in\na: k\le n, kM+r\in \bigcup_{j=1}^{J_n}A_j\}|\le\\
\frac{1}{n}\sum_{j=1}^{J_n}|\{k\in\na: k\le n, kM+r\in A_j\}|\le \frac{1}{n}((\sum_{j=1}^{J_n}\frac{n}{p'_j})+J_n)
\end{array}
$$
Thanks to our assumption on $p_j$,  $\frac{J_n}{n}\rightarrow 0$ as $n\rightarrow \infty$, thus
the assertion follows.
\end{proof}

\begin{lem}\label{staszek1}
In the family of all subsets of $\na$ (or of $\Z$) which have well defined density, density can be viewed
as a finitely additive probability measure. Then any finite collection of periodic sets with relatively
prime periods is stochastically independent.
\end{lem}

\begin{proof}
Since every periodic set with period $p$ decomposes as a disjoint union of finitely many arithmetic progressions with step $p$, it suffices to prove the lemma for arithmetic progressions (notice that
the density of a progression with step $p$ equals $\frac1p$). So, let $A_1, A_2,\dots,A_k$ be arithmetic progressions with steps $p_1,p_2,\dots,p_k$. We need to show that the density of their intersection
equals $\frac1{p_1p_2\dots p_k}$. This, however is obvious, because due to the relative primeness, this intersection is an arithmetic progression with step $p_1p_2\dots p_k$.
\end{proof}

\begin{lem}\label{staszek2}
For any $k\ge 1$ we have
$$
\prod_{j\ge k}(1-\tfrac 1{d_j})\ge 1-\tfrac1{\sqrt{d_k}}.
$$
\end{lem}

\begin{proof}
Recall that $d_j=q_j^2$, where $q_j$ denote the consecutive primes. Thus
\begin{multline*}
\prod_{j\ge k}(1-\tfrac 1{d_j})= \prod_{j\ge k}(1-\tfrac 1{q^2_j})=
\lim_{N\to\infty}\prod_{j=k}^N(1-\tfrac 1{q^2_j})\ge \lim_{N\to\infty}\prod_{n=q_k}^N(1-\tfrac 1{n^2}) = \\
\lim_{N\to\infty}\frac{q_k-1}{q_k}\cdot\frac{N+1}N = 1-\tfrac1{q_k}= 1-\tfrac1{\sqrt{d_k}}.
\end{multline*}
\end{proof}

\end{document}